\documentclass[preprint,10pt]{elsarticle}
\usepackage{amsmath}
\usepackage{amsthm}
\usepackage{amssymb}
\usepackage{graphicx}
\usepackage{epstopdf}
\usepackage[]{algorithm2e}
\DeclareMathOperator{\diff}{\mathrm{d} \! }

\newtheorem{theorem}{Theorem}
\newtheorem{corollary}{Corollary}

\begin{document}
	
\title{Full linear multistep methods as root-finders}

\author[tue]{Bart S. van Lith\fnref{fn1}}
\ead{b.s.v.lith@tue.nl}
\author[tue]{Jan H.M. ten Thije Boonkkamp}
\author[tue,phl]{Wilbert L. IJzerman}

\fntext[fn1]{Corresponding author}

\address[tue]{Department of Mathematics and Computer Science, Eindhoven University of Technology - P. O. Box 513, NL-5600 MB Eindhoven, The Netherlands.}
\address[phl]{Philips Lighting - High Tech Campus 44, 5656 AE, Eindhoven, The Netherlands.}

\begin{abstract}
Root-finders based on full linear multistep methods (LMMs) use previous function values, derivatives and root estimates to iteratively find a root of a nonlinear function. As ODE solvers, full LMMs are typically not zero-stable. However, used as root-finders, the interpolation points are convergent so that such stability issues are circumvented.  A general analysis is provided based on inverse polynomial interpolation, which is used to prove a fundamental barrier on the convergence rate of any LMM-based method. We show, using numerical examples, that full LMM-based methods perform excellently. Finally, we also provide a robust implementation based on Brent's method that is guaranteed to converge.
\end{abstract}


\maketitle
	
\section{Introduction}

Suppose we are given a sufficiently smooth nonlinear function $f: \mathbb{R} \to \mathbb{R}$ and we are asked to solve the equation
\begin{equation}\label{eq:root_finding_problem}
f(x) = 0.
\end{equation}
This archetypical problem is ubiquitous in all fields of mathematics, science and engineering. For example, ray tracing techniques in optics and computer graphics need to accurately calculate intersection points between straight lines, rays, and objects of varying shapes and sizes \cite{glassner,chaves}. Implicit ODE solvers are often formulated like \eqref{eq:root_finding_problem}, after which a root-finder of some kind is applied \cite{butcher}.\par
Depending on the properties of the function $f$, there are several methods that present themselves. Sometimes the derivative is not available for various reasons, in which case the secant method will prove useful. If higher-order convergence is desired, inverse quadratic interpolation may be used \cite{gautschi}. If the derivative of $f$ exists and is available, Newton's method is a solid choice, especially if $f$ is also convex. \par
Recently, a new interpretation of root-finding methods in terms of ODEs has been introduce by Grau-S\'anchez et al. \cite{grau_adams,grau_RK,grau_obreshkov}. Their idea is to consider the inverse function derivative rule as an ODE, so that any explicit ODE solver may be converted to a root-finding method. Indeed, Grau-S\'anchez et al. have successfully introduced root-finders based on Adams-type multistep and Runge-Kutta integrators. It goes without saying that only explicit ODE solvers can be usefully converted to root-finding methods. However, predictor-corrector pairs are possible, as those methods are indeed explicit.\par
We argue that the ODE approach can be interpreted as inverse interpolation with (higher) derivatives. Indeed, any linear integration method is based on polynomial interpolation. Thus, the ODE approach can be seen as a generalisation of inverse interpolation methods such as the secant method or inverse quadratic interpolation. The analysis can thus be combined into a single approach based on inverse polynomial interpolation.\par
Our main theoretical result is a theorem on the convergence rate of root-finders based on explicit linear multistep methods. We furthermore prove a barrier on the convergence rate of LMM-based root-finders. It turns out that adding a few history points quickly boosts the convergence rate close to the theoretical bound. However, adding many history points ultimately proves an exercise in futility due to diminishing returns in the convergence rate. Two LMM-based methods are constructed explicitly, one using two history points with a convergence rate of $1+\sqrt{3} \approx 2.73$ and another with three history points that converges with rate $2.91$.\par
In terms of the efficiency measure, defined as $p^\frac{1}{w}$ with $p$ the order and $w$ the number of evaluations \cite{gautschi}, our method holds up even compared to several optimal memoryless root-finders. The Kung-Traub conjecture famously states the optimal order of memoryless root-finders using $w$ evaluations is $2^{w-1}$, so that the efficiency is $2^\frac{w-1}{w}$. Our method requires two evaluations per iteration, provided function and derivative values are stored, leading to an efficiency of $1.74$ for the three-point method. As a consequence, anything up to an eighth-order method has a lower efficiency measure. For instance, optimal fourth-order methods with $\sqrt[3]{4} \approx 1.59$ \cite{Shengguo20091288,SOLEYMANI2012847,BEHL201589}, or eighth-order methods with $\sqrt[4]{8} \approx 1.68$ \cite{kim,CHUN201486,Lotfi2015}. Only 16th-order methods, e.g. Geum and Kim's \cite{GEUM20113278}, with $\sqrt[5]{16} \approx 1.74$ would in theory be more efficient.\par
Using several numerical examples, we show that the LMM-based methods indeed achieve this higher convergence rate. Furthermore, pathological examples where Newton's method fails to converge are used to show increased stability. We also construct a robust LMM-based method combined with bisection to produce a method that can been seen as an extension of Brent's \cite{brent}. Similar to Brent's method, whenever an enclosing starting bracket is provided, an interval $[a,b]$ with $f(a) f(b) <0$, our method is guaranteed to converge.\par
This article is organised in the following way. First, we find the convergence rate of a wide class of root-finders in Section~\ref{sec:barriers} and prove a barrier on the convergence rates. Next, in Section~\ref{sec:new_finders} we derive new root-finders based on full linear multistep methods and show that such methods are stable when the initial guess is sufficiently close to the root. After this, some results are presented in Section~\ref{sec:results} that verify our earlier theoretical treatment. Finally, we present our robust implementation in Section~\ref{sec:robust}, after which we give our conclusions in Section~\ref{sec:conclusions}.

\section{Barriers on LMM root-finders}\label{sec:barriers}
Root-finding methods based on the ODE approach of Grau-S\'anchez et al. can be derived by assuming that the function $f$ is sufficiently smooth and invertible in the vicinity of the root. Under these assumptions, the chain rule gives
\begin{equation}\label{eq:root_ODE}
\frac{\diff x}{\diff y} = [f^{-1}]^\prime (y) = \frac{1}{f^\prime\big(x\big)} = F(x),
\end{equation}
which we may interpret as an autonomous ODE for the inverse. Integrating \eqref{eq:root_ODE} from an initial guess $y_0 = f(x_0)$ to $y=0$ yields
\begin{equation}\label{eq:root_integral}
x(0) = x_0 + \int_{y_0}^{0} F\big( x(y) \big) \diff y,
\end{equation}
where $x(0)$ is the location of the root. Immediately, we see that applying the forward Euler method to \eqref{eq:root_ODE} gives Newton's method. From \eqref{eq:root_integral}, we see that the step size of the integrator should be taken as $0-y_0 = -f(x_0)$. However, Newton's method may also be interpreted as an inverse linear Taylor method, i.e., a method where the inverse function is approximated by a first-order Taylor polynomial. Indeed, any linear numerical integration method applied to \eqref{eq:root_ODE} can be interpreted as an inverse polynomial interpolation.\par
As such, explicit linear multistep methods applied to \eqref{eq:root_ODE} will also produce a polynomial approximation to the inverse function. Such a method has the form
\begin{equation}\label{eq:LMM_root-finder}
x_{n+s} + \sum_{k=0}^{s-1} a_k^{(n)} x_{n+k} = h_{n+s} \sum_{k=0}^{s-1} b_k^{(n)} F(x_{n+k}),
\end{equation}
where indeed $b_s^{(n)} = 0$, otherwise we end up with an implicit root-finder, which would not be very useful. The coefficients of the method, $\{a_k^{(n)}\}_{k=0}^{s-1}$ and $\{b_k^{(n)}\}_{k=0}^{s-1}$, will depend on the previous step sizes and will therefore be different each step. The step sizes are given by $h_{n+k} = y_{n+k} - y_{n+k-1}$, the differences in $y$-values. Since we wish to find the root, we set $y_{n+s} = 0$, leading to $h_{n+s} = y_{n+s} - y_{n+s-1} = - y_{n+s-1} $. Furthermore, the $y$-values are of course given by the function values of the root estimates, i.e.,
\begin{equation}
h_{n+k} =  f(x_{n+k}) -f(x_{n+k-1}) \quad  \text{for }k=1,\ldots,s-1.
\end{equation}
Like an ODE solver, we may use an implicit LMM in tandem with an explicit LMM to form a predictor-corrector pair, the whole forming an explicit method. Unlike an ODE solver, we may construct derivative-free root-finders based on the LMM approach by setting all $b_k^{(n)}=0$ for $k=0,\ldots s-1$ and for all $n>0$, e.g., the secant method. For an ODE solver this would obviously not make sense. Similar to ODE solvers, we may introduce higher derivatives by using
\begin{equation}\label{eq:general_LMM_root-finder}
x_{n+s} + \sum_{k=0}^{s-1} a_k^{(n)} x_{n+k} = h_{n+s} \sum_{k=0}^{s-1} b_k^{(n)} F(x_{n+k}) + h_{n+s}^2 \sum_{k=0}^{s-1} c_k^{(n)} F^\prime(x_{n+k})  + \ldots 
\end{equation}
The following theorem provides the maximal convergence rate for any method of the form \eqref{eq:general_LMM_root-finder}. Furthermore, it provides a fundamental barrier on the convergence rates of LMM-based root-finders. Under certain conditions, it also rewards our intuition in the sense that methods using more information converge faster to the root.\par
 Let us introduce some notation first. We denote $d$ the number of derivatives used in the method \eqref{eq:general_LMM_root-finder}. Higher derivatives of the inverse are found by iteratively applying the inverse function derivative rule. Methods defined by \eqref{eq:LMM_root-finder} are the special case of \eqref{eq:general_LMM_root-finder} with $d=1$. We also introduce coefficients $\sigma_k$ that indicate whether the coefficients $a_k^{(n)}$ are arbitrarily fixed from the outset or left free to maximise the order of convergence, i.e., $\sigma_k = 1$ if $a_k^{(n)} $ is free and $\sigma_k = 0$ otherwise.

\begin{theorem}\label{thm:convergence_bound}
	For simple roots, the convergence rate $p$ for any method of the form \eqref{eq:general_LMM_root-finder}, where the coefficients are chosen so as to give the highest order of convergence, is given by the largest real root of
	\begin{equation}\label{eq:convergence_root}
	p^s = \sum_{k=0}^{s-1} p^k(d + \sigma_k),
	\end{equation}
	for all $s \geq 1$ and $d \geq 1$, or $s \geq 2$ and $d=0$. The convergence rate is bounded by $p < d+2$.\\
	Additionally, if $\sigma_k = 1$ for all $k = 0,\ldots,s-1$, the convergence rates form a monotonically increasing sequence in $s$ and $p \to d+2$ as $s \to \infty$.
\end{theorem}

\begin{proof}
	1. Any method of the form \eqref{eq:general_LMM_root-finder} implicitly uses inverse polynomial (Hermite) interpolation applied to the inverse function $f^{-1}$, let us call the resulting interpolation $H$. Let $y_{n+k}$, $k = 0,\ldots,s-1$ be the interpolation points. At each point $y_{n+k}$, there are $d+\sigma_k$ values are interpolated, the inverse function value $x_k$ if $\sigma_k = 1$ and $d$ derivative values. Thus, the polynomial interpolation error formula gives
	\begin{equation*}
	f^{-1}(y) - H(y) = \frac{[f^{-1}]^{(N+1)}(\upsilon)}{(N+1)!} \prod_{k=0}^{s-1} (y - y_{n+k})^{d+\sigma_k},
	\end{equation*}
	where $\upsilon$ is in the interval spanned by the interpolation points and $N = sd + \sum_{k=0}^{s-1}\sigma_k$. The approximation to the root is then computed as $x_{n+s} = H(0)$. Let us denote the exact value of the root as $\alpha$, then
	\begin{equation*}
	|x_{n+s} - \alpha | =  \frac{|[f^{-1}]^{(N+1)}(\upsilon)|}{(N+1)!} \prod_{k=0}^{s-1} | y_{n+k}|^{d+\sigma_k}.
	\end{equation*}
	Define $\varepsilon_{n+k} =  x_{n+k} - \alpha $ and recognise that $f(x_{n+k}) = f(\alpha + \varepsilon_{n+k}) = f^\prime(\alpha) \varepsilon_{n+k} + \mathcal{O}(\varepsilon_{n+k}^2)$, where $f^\prime(\alpha) \neq 0$. Thus, we find
	\begin{equation*}
	| \varepsilon_{n+s}| \approx A_0 |\varepsilon_{n+s-1}|^{d+\sigma_{s-1}} \cdots | \varepsilon_{n}|^{d+\sigma_{0}},
	\end{equation*}
	where $A_0>0$ is a constant depending on $[f^{-1}]^{(N+1)}(\upsilon) $, $s$ and $f^\prime(\alpha)$. The error behaviour is of the form
	\begin{equation*}
	|\varepsilon_{l+1}| = C |\varepsilon_l|^p, \tag{$\ast$}
	\end{equation*}
	asymptotically as $l \to \infty$. Here, $C>0$ is a constant. Applying $(\ast)$ $s$ times on the left and $s-1$ times on the right-hand side leads to
	\begin{equation*}
	| \varepsilon_{n}|^{p^s} \approx A_1 | \varepsilon_{n}|^{\sum_{k=0}^{s-1} p^k(d+\sigma_k) },
	\end{equation*}
	where all the constants have been absorbed into $A_1$. Thus, \eqref{eq:convergence_root} is established.\par
	2. For methods that only use a single point, we have $s=1$ and $\sigma_0 = 1$ so that \eqref{eq:convergence_root} simplifies to $p = d + 1$. Hence, also in the case $s=1$, we have $p < d+2$.\par
	3. Finally, by its definition we can bound $\sigma_k \leq 1$, so that we obtain
	\begin{equation*}
	p^s \leq  (d+1) \sum_{k=0}^{s-1} p^k = (d+1) \frac{p^s-1}{p-1}.
	\end{equation*}
	Simplifying, we obtain
	\begin{equation*}
	p^{s+1}-(d+2)p^s + d+1 \leq 0. \tag{$\star$}
	\end{equation*}
	Note that $ p=1 $ is always a solution if we impose equality. However, the maximal convergence rate is given by the largest real root, so that we look for solutions $p>1$. Dividing by $p^s$ yields,
	\begin{equation*}
	p -( d+2) \leq - \frac{d+1}{p^s} <0,
	\end{equation*}
	which holds for all $s \geq 1$. Hence, we obtain $ p < d+2$.\par
	4. Suppose now that $\sigma_k = 1$ for all $k = 0,\ldots,s-1$, so that the convergence rate satisfies $(\star)$ with equality. Rewriting, we obtain
	\begin{equation*}
	d+2 - p = \frac{d+1}{p^s}.
	\end{equation*}
	Thus, the convergence rate is given by the intersection point of a straight line and an inverse power law, both as a function of $p$. First, we observe that there is always an intersection at $p=1$. Furthermore, the slope of the straight line is $-1$, while the slope of the inverse power law at $p=1$ is given by $-s(d+1)$. Therefore, the slope of the right-hand side is smaller than the slope of the left-hand side for all $s\geq 1$ and $d\geq 1$, or $d=0$ and $s\geq2$. Thus, to the right of $p=1$, the inverse power law is below the straight line. Combined with the fact that $ \frac{d+1}{p^s}$ is a convex function for $p>0$ and it goes asymptotically to zero, we see that the second intersection point exists and is unique. Fix $s$ and call the second intersection point, i.e., the convergence rate, $p_s$. Finally, we note that for $p>1$, we have
	\begin{equation*}
	\frac{d+1}{p^{s+1}} < \frac{d+1}{p^s},
	\end{equation*}
	so that $p_{s+1}$ will be moved towards the right compared to $p_s$. Thus, we find
	\begin{equation*}
	p_{s+1} > p_s
	\end{equation*}
	for all $d \geq 1$ and $s \geq 1$, or $d=0$ and $s \geq 2$. The result now follows from the Monotone Convergence Theorem \cite{adams}.
\end{proof}

From Theorem~\ref{thm:convergence_bound}, we find several special cases, such as the derivative-free interpolation root-finders, i.e., $d=0$. Note that derivative-free root-finders with $s=1$ simply do not exist, as then the inverse is approximated with a constant function.

\begin{corollary}
	Inverse polynomial interpolation root-finders, i.e., $d=0$ resulting in all $b_k^{(n)}=0$ in \eqref{eq:LMM_root-finder}, can attain at most a convergence rate that is quadratic. Their convergence rates are given by the largest real root of
	\begin{equation}\label{eq:inverse_polynomial_root}
	p^{s+1}-2p^s+1 =0,
	\end{equation}
	for all $s\geq 2$. The convergence rates are bounded by $p < 2$ and form a monotonically increasing sequence in $s$, with $p \to 2$ as $s \to \infty$.
\end{corollary}

\begin{proof}
	The coefficients $\{a_k^{(n)}\}_{k=0}^{s-1}$ are chosen to maximise the order of convergence, so that $\sigma_k = 1$ for all $k=0,\ldots,s-1$, while $d=0$, leading to
	\begin{equation*}
	p^s = \sum_{k=0}^{s-1} p^k = \frac{p^s-1}{p-1}.
	\end{equation*}
	Simplifying yields \eqref{eq:inverse_polynomial_root}. Furthermore, the condition that $\sigma_k = 1$ for all $k=0,\ldots,s-1$ is satisfied so that the convergence rates form a monotonically increasing sequence that converges to $d+2 = 2$.
\end{proof}

Inverse polynomial root-finders such as the secant method ($s=2$) or inverse quadratic interpolation ($s=3$) are derivative-free, so that their highest convergence rate is $2$ according to Theorem \ref{thm:convergence_bound}. The first few convergence rates for derivative-free inverse polynomial interpolation methods are presented in Table~\ref{tab:convergence_rates_interpolation}. The well known convergence rates for the secant method and the inverse quadratic interpolation method are indeed reproduced. As becomes clear from the table, the rates quickly approach $2$ but never quite get there. The increase in convergence rate becomes smaller and smaller as we increase the number of interpolation points.

\begin{table}[h]
	\centering
	\caption{The first few convergence rates for $s$ points using only function values.}
	\label{tab:convergence_rates_interpolation}
	\begin{tabular}{l|l}
		$s$   & $p$      \\ \hline
		$2$ & $1.62$    \\
		$3$ & $1.84$ \\
		$4$ & $1.92$ \\
		$5$ & $1.97$
	\end{tabular}
\end{table}

Next, we cover the Adams-Bashforth methods also discussed in \cite{grau_adams}. As ODE solvers, Adams-Bashforth methods are explicit integration methods that have order of accuracy $s$ \cite{hairer}. However, as Theorem~\ref{thm:convergence_bound} suggests, as root-finders they will have a convergence rate that is smaller than cubic, since $d=1$. In fact,  the convergence rate of Adams-Bashforth root-finders is bounded by $\frac{3+\sqrt{5}}{2} = 2.62$ as was proven by Grau-S\'anchez et al. \cite{grau_adams}. The following corollary is a generalisation of their result.

\begin{corollary}
	The Adams-Bashforth root-finder methods with $s\geq 2$ exhibit convergence rates given by the largest real root of
	\begin{equation}\label{eq:AB_root}
	p^{s+1} - 3 p^s + p^{s-1} + 1 =0,
	\end{equation}
	for all $s \geq 1$. The convergence rates are bounded by $p < \frac{3 + \sqrt{5}}{2}$ and form a monotonically increasing sequence in $s$, with $p \to  \frac{3 + \sqrt{5}}{2}$ as $s \to \infty$.
\end{corollary}

\begin{proof}
	1. Adams-Bashforth methods have $a_k^{(n)}=0$ for $k=0,\ldots,s-2$, resulting in $\sigma_k = 0$ for $k = 0,\ldots,s-2$ and $\sigma_{s-1} = 1$. We may write $\sigma_k$ for simplicity as a Kronecker delta, i.e., $\sigma_k = \delta_{k,s-1}$. Furthermore, the methods use a single derivative of $f^{-1}$ so that $d=1$. The $s=1$ method is equal to Newton's method, which has a quadratic convergence rate. For $s \geq 2$, we find from Theorem~\ref{thm:convergence_bound} that
	\begin{equation*}
	p^s =p^{s-1} + \sum_{k=0}^{s-1} p^k = p^{s-1} + \frac{p^{s}-1}{p-1} .
	\end{equation*}
	Simplifying yields \eqref{eq:AB_root}. Again, we assume that $p>1$ and we divide by $p^{s-1}$, so that
	\begin{equation*}
	p^2 - 3p + 1 =- \frac{1}{p^{s-1}} < 0,
	\end{equation*}
	which holds for all $s \geq 2$. This implies that $p < \frac{3 + \sqrt{5}}{2}$ for all $s \geq 2$.\par
	2. The proof that the convergence rates make up a monotonically increasing sequence is similar to the one given for Theorem~\ref{thm:convergence_bound}. First, we note that \eqref{eq:AB_root} always has a root at $p=1$. Next, we rewrite it to read
	\begin{equation*}
	3-p = \frac{1}{p} + \frac{1}{p^s}.
	\end{equation*}
	The left-hand side is a straight line with slope $-1$ while the right-hand side is a convex function that has slope $-3$ at $p=1$. Thus, to the right of $p=1$, the convex function is below the straight line. Furthermore, the right-hand side goes asymptotically to zero for large $p$. This implies that there is a unique intersection point with $p>1$, which is the convergence rate. For fixed $s$, call the second intersection point $p_s$. Finally, we note that
	\begin{equation}
	\frac{1}{p} + \frac{1}{p^{s+1}} < \frac{1}{p} + \frac{1}{p^s},
	\end{equation}
	for $p>1$, from which we see that the intersection point is moved to the right for larger $s$, i.e., $p_{s+1}>p_s$.  The result again follows from the Monotone Convergence Theorem.
\end{proof}

The first few convergence rates for the Adams-Bashforth root-finder methods are given in Table~\ref{tab:convergence_rates_AB} and agree with the rates found by Grau-S\'anchez et al. As becomes clear from the table, the convergence rates quickly draw near the bound of $2.62$. Yet again we are met with steeply diminishing returns as we increase the number of history points $s$.

\begin{table}[h]
	\centering
	\caption{The first few convergence rates for Adams-Bashforth root-finder method using $s$ points.}
	\label{tab:convergence_rates_AB}
	\begin{tabular}{l|l}
		$s$   & $p$      \\ \hline
		$1$ & $2$ \\
		$2$ & $2.41$    \\
		$3$ & $2.55$ \\
		$4$ & $2.59$ \\
		$5$ & $2.61$
	\end{tabular}
\end{table}

The Adams-Bashforth root-finder methods cannot attain a convergence rate higher than $2.62$, which is still some way off the cubic bound given by Theorem~\ref{thm:convergence_bound}. However, the theorem also tells us that using full linear multistep methods will result in convergence rates closer to cubic, at least in the limit of large $s$. For ODE solvers, trying to obtain a higher convergence rate by increasing the number of points often leads to instabilities. Indeed, the stability regions of many LMMs tend to shrink as the order is increased \cite{butcher}. In general, polynomial interpolation on equispaced points is unstable, e.g., Runge's phenomenon \cite{quarteroni}.\par
We must note that the stability issues in ODE solvers arise from the fact that polynomial interpolation is applied on an equispaced grid. Root-finders are designed to home in on a root, and when convergent, the step sizes will decrease rapidly. Runge's phenomenon can be countered by placing more nodes closer to the boundary, for example Gau{\ss} or Lebesgue nodes. As we will see, inverse polynomial interpolation is stable on the set of points generated by the root-finder itself, provided the starting guess is sufficiently close.\par
Let us inspect the convergence rates of different LMM-based root-finders using Theorem~\ref{thm:convergence_bound}, see Table~\ref{tab:convergence_rates_derivatives}. These convergence rates are computed under the assumption that all derivatives and point values are used, i.e., $\sigma_k = 1$ for $k=0,\ldots,s-1$ in Theorem~\ref{thm:convergence_bound}. The convergence rate of a $d$-derivative method can be boosted by at most $1$, and the table shows that this mark is attained very quickly indeed. Adding a few history points raises the convergence rate significantly, but finding schemes with $s>3$ is likely to be a waste of time.

\begin{table}[h]
	\centering
	\caption{The first few convergence rates for $s$ points (vertical) using function values and the first $d$ derivatives (horizontal).}
	\label{tab:convergence_rates_derivatives}
	\begin{tabular}{l|llll}
		$s\backslash d$ & $1$    & $2$    & $3$    & $4$    \\ \hline
		$1$   & $2$    & $3$    & $4$    & $5$    \\
		$2$   & $2.73$ & $3.79$ & $4.82$ & $5.85$ \\
		$3$   & $2.91$ & $3.95$ & $4.97$ & $5.98$ \\
		$4$   & $2.97$ & $3.99$ & $4.99$ & $5.996$
	\end{tabular}
\end{table}

Thus, provided that the root-finders are stable, a higher convergence rate can be achieved by adding history points, as well as derivative information.

\section{Full LMM-based root-finders}\label{sec:new_finders}
Let us investigate full LMM-based root-finders that use a single derivative, thus methods of the form \eqref{eq:LMM_root-finder}. The current step size is then given by $h_{n+s} = -f(x_{n+s-1})$. Let us define $q_k^{(n)}$ as
\begin{equation}
q_k^{(n)} = \frac{f(x_{n+k})}{f(x_{n+s-1})}, \quad k = 0,\ldots ,s-2,
\end{equation}
so that  $h_{n+s} q_k^{(n)} = -f(x_{n+k})$ is the total step between $y_{n+k}$ and $y_{n+s}=0$. The Taylor expansions of $x(y_{n+k})$ and $x^\prime(y_{n+k})$ about $y_{n+s}$ are then given by
\begin{subequations}
	\begin{align}
	x(y_{n+k}) = x(y_{n+s}) +\sum_{m=1}^\infty \frac{1}{m!} (-h_{n+s} q_k)^m x^{(m)} (y_{n+s}),\\
	x^\prime(y_{n+k}) = x^\prime(y_{n+s}) + \sum_{m=1}^\infty \frac{1}{m!} (-h_{n+s} q_k)^m x^{(m+1)} (y_{n+s}),
	\end{align}
\end{subequations}
where we have dropped the superscript $(n)$ for brevity. Substituting these into \eqref{eq:LMM_root-finder}, we obtain
\begin{equation}
\begin{aligned}
&x(y_{n+s}) \left[1 + \sum_{k=0}^{s-1} a_k \right]- h_{n+s} x^\prime(y_{n+s}) \left[ \sum_{k=0}^{s-1} a_k q_k + b_k \right]\\
&+ \sum_{m=2}^{\infty} \frac{1}{(m-1)!}  (-h_{n+s})^m x^{(m)}(y_{n+s}) \sum_{k=0}^{s-1} \left[  \tfrac{1}{m} q_k^m a_k  +  q_k^{m-1} b_k  \right]=0.
\end{aligned}
\end{equation}
The consistency conditions then are given by
\begin{subequations}\label{eq:LMM_consistency}
	\begin{align}
	&\sum_{k=0}^{s-1} a_k= -1,\\
	&\sum_{k=0}^{s-1} a_k q_k + b_k = 0.
	\end{align}
\end{subequations}
This gives us two equations for $2s$ coefficients, so that we can eliminate another $2s-2$ leading order terms, resulting in the conditions
\begin{equation}\label{eq:LMM_order_equations}
\sum_{k=0}^{s-1} \frac{q^m_k}{m} a_k + q_k^{m-1} b_k = 0,
\end{equation}
where $m = 2,\ldots,2s-1$.

\subsection{The $s=2$ method}
\noindent The $s=2$ LMM-based method is given by
\begin{equation}\label{eq:inverse_cubic_LMM}
x_{n+2} + a_1 x_{n+1} + a_0 x_n = h_{n+2} \Big( b_1 F(x_{n+1}) + b_0 F(x_{n})  \Big),
\end{equation}
where we have again suppressed the superscript $(n)$ on the coefficients. Here, $h_{n+2} = - f(x_{n+1})$ so that we may write $q = q_0$, i.e.
\begin{equation}
q = \frac{f(x_n)}{f(x_{n+1})}.
\end{equation}
Applying \eqref{eq:LMM_consistency} and \eqref{eq:LMM_order_equations}, we find a set of linear equations, i.e.,
\begin{subequations}
\begin{align}
 a_1 + a_0 = -1,\\
 a_1 + q a_0 +b_1 + b_0 = 0,\\
\tfrac{1}{2}a_1 + \tfrac{1}{2} q^2  a_0 +b_1 + q b_0 = 0,\\
\tfrac{1}{3} a_1 + \tfrac{1}{3} q^3 a_0 + b_1 + q^2 b_0 = 0.
\end{align}
\end{subequations}
These equations may be solved, provided $q \neq 1$, to yield
\begin{subequations}\label{eq:LMM_coefficients}
	\begin{align}
 a_0 &= \frac{1-3q}{(q-1)^3} & 	a_1 &= -1-a_0, \\
b_0 &= \frac{q}{(q-1)^2} &	b_1 &= q b_0 .
	\end{align}
\end{subequations}
The condition $q\neq 1$ is equivalent to $f(x_{n+1}) \neq f(x_n)$. This condition is not very restrictive, as stronger conditions are needed to ensure convergence.\par
The above method may also be derived from the inverse polynomial interpolation perspective, using the ansatz
\begin{equation}\label{eq:inverse_cubic_interpolation}
H(y) = h_3 \big(y-f(x_{n+1}) \big)^3 + h_2 \big(y-f(x_{n+1}) \big)^2 + h_1 \big(y-f(x_{n+1}) \big) + h_0,
\end{equation}
where $h_i$, $i=0,1,2,3$ are undetermined coefficients. The coefficients are fixed by demanding that $H$ interpolates $f^{-1}$ and its derivative at $y=f(x_{n+1})$ and $y=f(x_n)$, i.e.,
\begin{subequations}\label{eq:inverse_cubic_conditions}
\begin{align}
H\big(f(x_n)\big) &= x_n, \\
H\big(f(x_{n+1})\big) &= x_{n+1},\\
H^\prime \big(f(x_n)\big) &= \frac{1}{f^\prime(x_n)},\\
H^\prime \big(f(x_{n+1})\big) &= \frac{1}{f^\prime(x_{n+1})}.
\end{align}
\end{subequations}
Solving for $h_i$, $i=0,1,2,3$ and setting $y=0$, we find the same update $x_{n+2}$ as \eqref{eq:inverse_cubic_LMM}.\par
The stability of the $s=2$ LMM method depends on the coefficients of the LMM in much the same way as an ODE solver. Indeed, we can set the sequence $\tilde{x}_n = x_n + z_n$ where $x_n$ is the sequence generated by exact arithmetic we wish to find while $z_n$ is a parasitic mode. It can be shown that the parasitic mode satisfies
\begin{equation}
z_{n+2} + a_1 z_{n+1} + a_0 z_n = 0,
\end{equation}
so that it may grow unbounded if the roots are greater than $1$ in modulus. Using the ansatz $z_n = B \lambda^n$, we find the characteristic polynomial of the $s=2$ method, i.e.,
\begin{equation}
\rho(\lambda) = \lambda^2 - \lambda \left( 1+a_0  \right) + a_0 = (\lambda -1) \left(\lambda - a_0\right),
\end{equation}
where the roots cans simply be read off. Stability of the root-finder is ensured if the stability polynomial of the method has a single root with $\lambda = 1$, while the other roots satisfy $|\lambda| <1$. This property is called zero-stability for linear multistep ODE solvers. Thus, to suppress parasitic modes we need
\begin{equation}
|a_0| = \left| \frac{1-3q}{(q-1)^3} \right| <1.
\end{equation}
This reduces to $q$ being either $q<0$, or $q>3$, so that $|q|>3$ is a sufficient condition. Thus, if the sequence $\{ |f(x_n)| \}_{n=1}^\infty$ is decreasing fast enough, any parasitic mode is suppressed. We may estimate $q$ as a ratio of errors, since $f(x_n) = f^\prime( \alpha) \varepsilon_n + \mathcal{O}(\varepsilon_n^2)$, so that
\begin{equation}
q \approx \frac{\varepsilon_n}{\varepsilon_{n+1}}.
\end{equation}
Using $\varepsilon_{n+1} = C \varepsilon_n^p$ with $p=1+\sqrt{3}$, we find that
\begin{equation}
\varepsilon_n < C_1,
\end{equation}
with $C_1 = \left( \frac{1}{3 C} \right)^\frac{1}{\sqrt{3}} $. We conclude that the method will be stable if the initial errors are smaller than the constant $C_1$, which depends on the details of the function $f$ in the vicinity of the root. This condition translates to having the starting values sufficiently close to the root. This is a rather typical stability condition for root-finders.

\subsection{$s=3$ method}
We may again apply \eqref{eq:LMM_consistency} - \eqref{eq:LMM_order_equations} to find a method with $s=3$, this time we have $6$ coefficients, given by
\begin{subequations}
	\begin{align}
	a_0 &= \frac{q_1^2( q_0 (3 + 3q_1 - 5 q_0) -q_1  )}{(q_0-1)^3(q_0-q_1)^3}, & b_0 &= \frac{q_0 q_1^2}{(q_0-1)^2(q_0-q_1)^2}, \\
	a_1 &=  \frac{q_0^2( q_1 (5 q_1- 3q_0 - 3 ) + q_0  )}{(q_1-1)^3(q_0-q_1)^3},  &  b_1 &= \frac{q_0^2 q_1}{(q_0-q_1)^2 (q_1 - 1)^2}, \\
	a_2 &= \frac{q_0^2 q_1^2( 3q_1 - q_0(q_1-3)-5 )}{(q_0-1)^3 (q_1-1^3) } , & b_2 &= \frac{q_0^2 q_1^2}{(q_0-1)^2(q_1-1)^2},
	\end{align}
\end{subequations}
where $q_0 = \frac{f(x_n)}{f (x_{n+2}) }$ and $q_1 = \frac{f(x_{n+1})}{f (x_{n+2}) }$. Here, we have the conditions $q_0 \neq 1$ and $q_1 \neq 1$, reducing to the condition that all $y$-values must be unique. Again, this condition is not very restrictive for reasons detailed above.\par
Methods with a greater number of history points are possible, however, the gain in convergence rate from $s=3$ to $s=4$ is rather slim, as indicated by Table~\ref{tab:convergence_rates_derivatives}. If such methods are desirable, they can be derived by selecting coefficients that satisfy \eqref{eq:LMM_consistency} - \eqref{eq:LMM_order_equations}.

\section{Results}\label{sec:results}
Like the secant method, the $s=2$ full LMM root-finding method needs two starting points for the iteration. However, as the analytical derivative is available, we choose to simply start Newton's method with one point, say $x_0$, giving $x_1$. The $s=3$ method needs three starting values, therefore the next value $x_2$ is obtained from the $s=2$ LMM method. The LMM-based methods can be efficiently implemented by storing the function value and derivative value of the previous step, thus resulting in a need for only one function and one derivative evaluation per iteration.\par
The efficiency measure is defined as $p^{\frac{1}{w}}$ with $p$ the order of convergence and $w$ the number of evaluations per iteration \cite{gautschi}. Assuming the function itself and the derivative cost the same to evaluate, the $s=3$ LMM-based method has an efficiency measure of $\sqrt{2.91} \approx 1.71$. Compared to Newton's method, with an efficiency measure of $\sqrt{2} \approx 1.41$,  this is certainly an improvement. 

\subsection{Numerical examples}
Here, we provide a number of test cases and show how many iterations LMM-based root-finders take versus the number needed by Newton's method, see Table~\ref{tab:numerical_examples_newton_hybrid}. We have used a selection of different test cases with polynomials, exponentials, trigonometric functions, square roots and combinations thereof. For each of the test problems shown, the methods converged within a few iterations. Some problems were deliberately started near a maximum or minimum to see the behaviour when the derivatives are small.\par
The test computations were performed using the variable-precision arithmetic of MATLAB's Symbolic Math Toolbox. The number of digits was set to 300 while the convergence criterion used was
\begin{equation}
|x_{l+1} - x_l| \leq 10^{-\eta},
\end{equation}
with $\eta = 250$. The numerical convergence rates were computed with the error behaviour
\begin{equation}\label{eq:error_behavior}
|\varepsilon_{l+1}| = C | \varepsilon_{l}|^p,
\end{equation}
asymptotically as $l \to \infty$. The limiting value of the estimates for $p$ is displayed in Table~\ref{tab:numerical_examples_newton_hybrid}.

\begin{table}[h]
\centering
\caption{Test cases with iterations taken for Newton's method (subscript $N$) and the LMM-based method with $s=2$ (subscript $2$) and $s=3$ method (subscript $3$). }
\label{tab:numerical_examples_newton_hybrid}
\begin{tabular}{lrl|lllll}
	function & root & $x_0$ & $\#$its$_N$ & $\#$its$_{2}$ & $\#$its$_{3}$ &  $p_{2}$ & $p_{3}$\\ \hline
	$x + e^x$  & $-0.57$ & $1.50$ & 11 & 8 & 8 &  $2.73$ & $2.93$ \\
	$\sqrt{x} -\cos(x)$ & $0.64$ & $0.50$ & 9 & 7 & 8 &  $2.74$ & $2.91$ \\
	$e^x -x^2+3x-2$ & $0.26$ & $0.00$ & 10 & 8 & 7 &  $2.72$ & $2.94$ \\
	$x^4-3x^2-3$ & $1.95$ & $1.30$ & 17 & 14 & 14 &  $2.73$ & $2.92$ \\
	$x^3-x-1$ & $1.32$ & $1.00$ & 12 & 9 & 9 &  $2.73$ & $2.64$ \\
	$e^{-x}-x^3$ & $0.77$ & $2.00$ & 13 & 10 & 10 & $2.73$ & $2.92$ \\
	$5\big(\sin(x)+\cos(x)\big)-x$ & $2.06$ & $1.50$ & 11 & 9 & 9 &  $2.73$ & $2.92$ \\
	$x-\cos(x)$ & $0.74$ & $1.00$ & 9 & 7 & 7 &  $2.72$ & $2.93$ \\
	$\log(x-1)+\cos(x-1)$ & $1.40$ & $1.60$ & 12 & 9 & 9 &  $2.73$ & $2.92$ \\
	$\sqrt{1+x} -x$ & $1.62$ & $1.00$ & 9 & 7 & 7 &  $2.73$ & $2.92$ \\
	$\sqrt{e^x-x} - 2x$ & $0.54$ & $1.00$ & 11 & 8 & 7 &  $2.73$ & $2.92$  \\ \hline \hline
\multicolumn{2}{l}{Total number of iterations} & & $124$ & $96$ & $95$ & &   \\
\end{tabular}
\end{table}

Overall, Newton's method consistently displayed a quadratic convergence rate, which is the reason we did not display it in the table. The LMM-based methods, on the other hand, generally have a higher convergence rate that may vary somewhat from problem to problem. This is due to the fact that the step sizes may vary slightly while the convergence rates of the LMM-based methods only holds asymptotically, even with so many digits.

\subsection{Pathological functions}
A classical example of a pathological function for Newton's method is the hyperbolic tangent $\tanh(x)$. Students often believe that Newton's method converges for any monotone function until they are asked to find the root of $\tanh(x)$. Hence, we have used this function as a test case using standard double precision floating point arithmetic and a convergence criterion reading
\begin{equation}\label{eq:double_precision_criterion}
|x_{l+1} - x_l| \leq 2 \epsilon,
\end{equation}
with $\epsilon$ the machine precision. Newton's method fails to converge for starting values with approximately $|x_0|\geq1.089$, see Table~\ref{tab:tanh_test}. Our $s=2$ LMM-based method extends this range somewhat more and converges for any starting value with roughly $|x_0| \leq 1.239$. The behaviour of the $s=3$ LMM-based method is similar, though it does take two more iterations to converge.

\begin{table}[]
	\centering
	\caption{Convergence history for $\tanh(x)$ of the three methods: Newton (subscript $N$) and the LMM-based methods with $s=2$ and $s=3$. Note that the root of $\tanh(x)$ is at $x=0$.}
	\label{tab:tanh_test}
	\begin{tabular}{r|r|r}
		$x_N$              & $x_{(s=2)}$              & $x_{(s=3)}$              \\ \hline
		$1.239$ & $1.239$ & $1.239$ \\
		$-1.719$ & $-1.719$ & $-1.719$ \\
		$6.059$ & $0.8045$ & $0.8045$ \\
		$-4.583 \cdot 10^{4}$ & $0.7925$ & $-0.6806$ \\
		$\mathrm{Inf}$ & $-0.7386$ & $1.377$ \\
		 & -$6.783 \cdot 10^{-3}$ & $-0.7730$ \\
		 & $9.323 \cdot 10^{-6}$ & $3.466 \cdot 10^{-2}$ \\
		 & $|x_{(s=2)}|< \epsilon$ & $-3.032 \cdot 10^{-4}$ \\
		 & & $1.831 \cdot 10^{-11}$ \\
		 &  & $|x_{(s=3)}|<\epsilon$ \\         
	\end{tabular}
\end{table}

Starting at $x_0 = 1.239$, Newton's method diverges quickly, returned as $-\mathrm{inf}$ after only 4 iterations. The LMM-based method, on the other hand, bounces around positive and negative values for about 5 iterations until it is close enough to the root. After that, the asymptotic convergence rate sets in and the root is quickly found, reaching the root within machine precision at 7 iterations.\par
Donovan et al.\cite{donovan_miller_moreland} developed another test function for which Newton's method fails, it gives a false convergence result to be precise. The test function is given by
\begin{equation}\label{eq:DMM_test_function}
h(x) = \sqrt[3]{x}e^{-x^2},
\end{equation}
which is, in fact, infinitely steep near the root $x=0$, yet smooth, see Figure~\ref{fig:DMM_test}. Again, we used double precision arithmetic and \eqref{eq:double_precision_criterion} as a stopping criterion. Newton's method diverges for any starting value except the exact root itself. However, Newton's method eventually gives a false convergence result as the increment $|x_{l+1}-x_l|$ falls below the tolerance. The $s=2$ and $s=3$ LMM-based methods converge when starting with $|x_0| \leq 0.1147$ for this problem, see Table~\ref{tab:DMM_test}.

\begin{figure}[h]
\centering
\includegraphics[width=\textwidth]{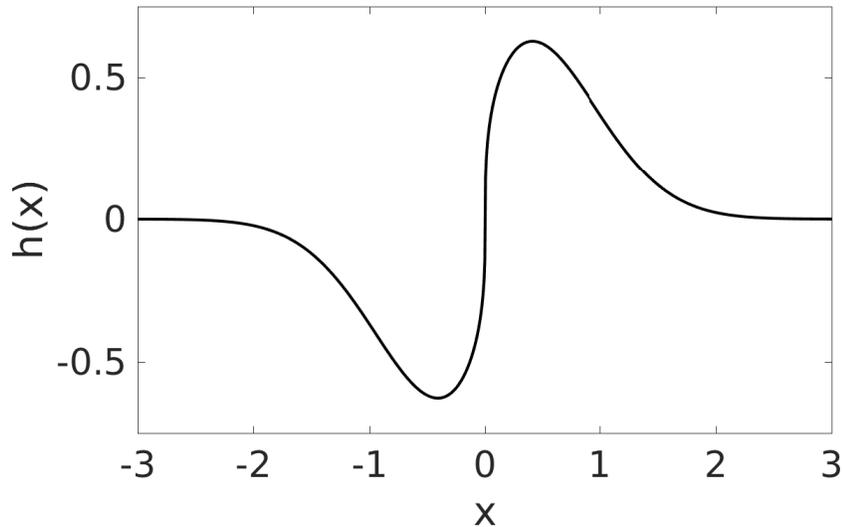}
\caption{Pathological test function $h(x)$ from \eqref{eq:DMM_test_function}.}
\label{fig:DMM_test}
\end{figure}

\begin{table}
	\centering
	\caption{Convergence history for $h(x)$ from \eqref{eq:DMM_test_function} of the three methods: Newton (subscript $N$) and the LMM-based methods with $s=2$ and $s=3$. Note that the root is at $x=0$.}
	\label{tab:DMM_test}
	\begin{tabular}{r|r|r}
			$x_N$              & $x_{(s=2)}$              & $x_{(s=3)}$              \\ \hline
			$0.1147$ & $0.1147$ 									& $0.1147$ \\
			$0.2589$ & $-0.2589$ 									& $-0.2589$ \\
			$1.0402$ & $0.1016$ 									& $0.1016$ \\
			$1.6084$ & $9.993 \cdot 10^{-2}$ 		  & $-5.648 \cdot 10^{-2}$ \\
			$1.9407$ & $-0.2581$ 									& $0.1959$ \\
			$2.2102$ & $9.840 \cdot 10^{-2}$ 		  & $-0.1611$ \\
			$2.4445$ & $9.810 \cdot 10^{-2}$  		  & $5.021 \cdot 10^{-2}$ \\
			$2.6549$ & $-0.2344$ 									& $-7.190 \cdot 10^{-2}$ \\
			$2.8478$ & $6.602 \cdot 10^{-2}$ 		  & $4.947 \cdot 10^{-2}$ \\
			$3.0270$ & $6.021 \cdot 10^{-2}$ 		  & $-3.777 \cdot 10^{-3}$ \\
			$3.1953$ & $-4.939 \cdot 10^{-2}$ 		  & $3.027 \cdot 10^{-4}$ \\
			$3.3543$ & $-4.019 \cdot 10^{-4}$ 		  & $-6.875 \cdot 10^{-6}$ \\
			$3.5056$ & $1.288 \cdot 10^{-4}$ 		  & $1.216 \cdot 10^{-9}$ \\
			$3.6502$ & $2.028 \cdot 10^{-10}$ 		& $-4.652 \cdot 10^{-15}$ \\
			$3.7889$ & $-5.308 \cdot 10^{-15}$ 		& $|x_{(s=2)}|< \epsilon$ \\
			$3.9225$ & $|x_{(s=2)}|< \epsilon$ &  \\
		\end{tabular}
\end{table}

Starting at the maximal $x_0=0.1147$ for instance, the LMM-based methods bounce several times between positive and negative $x$-values without making much headway. After that, the root is close enough and the asymptotic convergence rate sets in, reaching the root to within machine precision in a few steps.\par
We believe the reason that the LMM-based method has increased stability is due to the fact that it uses two points to evaluate the function and its derivative. In both cases, the iterations jump between positive and negative values, enclosing the root. In this fashion, the LMM-based method acts much like the \textit{regula falsi} method. Once the iterates are close enough to the root, the asymptotic convergence rate sets in and the iterates converge in but a few steps.

\section{A robust implementation}\label{sec:robust}
As with most open root-finding algorithms, the conditions under which the method is guaranteed to converge are rather restrictive. Therefore, we have designed a bracketing version of the LMM-based method that is guaranteed to converge. The algorithm is based on Brent's method, using similar conditions to catch either slow convergence or runaway divergence. This version of the LMM-based method does, however, require an enclosing bracket $[a,b]$ on which the function changes sign, i.e., $f(a) f(b)<0$. Alternatively, such a method can start out as an open method, switching to the bracketing method when a sign change is detected.\par
The algorithm consists of a cascade of methods increasing in accuracy but decreasing in robustness, similar to Brent's method. At the lowest level stands the most robust method, bisection, guarding against steps outside the current search bracket. Bisection is guaranteed to converge, but does so rather slowly. On the highest level we use the full $s=3$ LMM-based method discussed in the previous section. Thus, in the best possible case, the method will converge with a rate of $2.91$. The method is, by virtue of the bisection method, guaranteed to converge to a root.\par
Like Brent's method and Dekker's method, the LMM-based method keeps track of three points $a$, $b$ and $c$. Here, $b$ is the best estimate of the root so far, $c$ is the previous value for $b$ while $a$ is the contrapoint so that $\mathrm{int}[a,b]$ encloses the root. Ideally, all three values are used to compute the next value for $b$. However, extra conditions are added to ensure the inverse actually makes sense on the interval $\mathrm{int}[a,c]$.\par
Consider the case where the sign of $f^\prime(c)$ is not equal to the sign of $\frac{f(b)-f(a)}{b-a}$, but the sign of $f^\prime(b)$ is. It follows that there is an extremum between $b$ and $c$, and the inverse function does not exist in the interval $\mathrm{int}[a,c]$, leading to an error if we were to compute the inverse interpolation.\par 
Thus, the following condition should be applied to each derivative value: the sign of the derivative needs to be the same as the sign of the secant slope on $\mathrm{int}[a,b]$, i.e.,
\begin{equation}\label{eq:derivative_condition}
\mathrm{sgn }\left( f^\prime(z) \right) = \mathrm{sgn } \left( \frac{f(b)-f(a)}{b-a} \right),
\end{equation}
where $z=a,b,c$. Only when the derivative at a point satisfies \eqref{eq:derivative_condition} can the derivative sensibly contribute to the inverse. Otherwise the derivative information should be discarded, leading to lower-order interpolation. If all derivatives are discarded, the resulting interpolation is inverse quadratic or the secant method.\par
Ultimately, the method provides an interval on which the function $f$ changes sign with a relative size of some given tolerance $\delta$, i.e.,
\begin{equation}\label{eq:brent_criterion}
|a - b| \leq \delta |b|.
\end{equation}
We shall use $\delta = 2 \epsilon$ in all our examples, with $\epsilon$ the machine precision.  As an input, the algorithm has $f$, $f^\prime$, $a$ and $b$ such that $f(a)f(b)<0$. The algorithm can be described in the following way:

\begin{enumerate}
\item If all three function values are different, use $s=3$, otherwise use $s=2$.

\item Check the sign of the derivatives at points $a$, $b$ and $c$, include the derivatives that have the proper sign.

\item If the interpolation step is worse than a bisection step, or outside the interval $[a,b]$, use bisection.

\item If the step is smaller than the tolerance, use the tolerance as step size.

\item If the convergence criterion is met, exit, otherwise go to 1.

\end{enumerate}

The first step determines the number of history points that can be used. The second step determines which derivative values should be taken into account. In effect, only the second step is essentially different from Brent's method, with all the following steps exactly the same \cite{brent}. The extra conditions on the derivatives gives rise to a selection of 12 possible root-finders, including inverse quadratic interpolation and the secant method. Our method can therefore be seen as building another 10 options on top of Brent's method. Naturally, sufficiently close to the root, the derivative conditions will be satisfied at all three points and the method will use the full LMM method with $s=3$.

\subsection{Comparison with Brent's method}
Here we give a few examples of the robust LMM-based root-finding algorithm discussed above compared to Brent's method. As a performance measure, we use the number of iterations. Standard double precision arithmetic is employed, as that provides sufficient material for comparison. For both methods, the stopping criterion is given by \eqref{eq:brent_criterion}, the relative size of the interval must be sufficiently small.

\begin{table}[h]
	\centering
	\caption{Test cases with iterations taken for Brent's method and the LMM-based method. Subscript $B$ represents Brent while subscript $LMM$ represents the LMM-based method.}
	\label{tab:examples_brent_hybrid}
	\begin{tabular}{lrl|ll}
function    				&		root      & $[a,b]$		& $\#$its$_B$ & $\#$its$_{LMM}$  \\ \hline
$x + e^x$  				  	&$-0.57$&  $[-1,1]$		& $6$    				 & $4$    			  \\
$\sqrt{x} -\cos(x)$&$0.64$& $[0,2]$      & $8$	          	 		& $4$			 \\
$e^x -x^2+3x-2$		&$0.26$& $[-1,1]$		& $5$					& $3$			\\
$x^4-3x^2-3$		  &	$1.95$	& $[1,3]$	   		& $10$		  		& $8$		 \\
$x^3-x-1$				  &	$1.32$	& $[0,2]$		   & $29$			    & $6$			\\
$e^{-x}-x^3$				  &	$0.77$  & $[0,2]$	  		& $9$				  & $4$				 \\
$5\big(\sin(x)+\cos(x)\big)-x$&$2.06$ &$[0,4]$&$45$	  & $6$ \\
$x-\cos(x)$				& $0.74$	& $[0,1]$		  & $7$				  & $3$				 \\
$\log(x-1)+\cos(x-1)$&$1.39$&$[1.2,1.6]$& $31$		& $4$			 \\
$\sqrt{1+x} -x$		&		$1.62$ & $[0,2]$		& $5$				& $3$				\\
$\sqrt{e^x-x} - 2x$&$0.54$& $[-1,2]$		 & $9$				 & $4$				\\ \hline \hline
Total number of iterations & & & $164$ & $49$ \\
Total number of function evaluations &	&&		$164$  & $98$
	\end{tabular}
\end{table}

Table~\ref{tab:examples_brent_hybrid} shows that for most functions, both Brent's method and the LMM-based method take a comparable number of iterations. However, in some cases, the difference is considerable. In the worst case considered here, Brent's method takes $7.5$ times as many iterations to converge. In terms of efficiency index, Brent's method should be superior with an efficiency index of $1.84$ against $1.71$ of the LMM-based method. Taken over the whole set of test functions, however, Brent's method takes more than three times as many iterations in total, leading to a significant increase in function evaluations. We conclude therefore that practically, the LMM-based root-finder is a better choice.

\section{Conclusions}\label{sec:conclusions}
We have discussed root-finders based on full linear multistep methods. Such LMM-based methods may be interpreted as inverse polynomial (Hermite) interpolation methods, resulting in a simple and general convergence analysis. Furthermore, we have proven a fundamental barrier for LMM-based root-finders: their convergence rate cannot exceed $d+2$, where $d$ is the number of derivatives used in the method.\par
The results indicate that compared to the Adams-Bashforth root-finder methods of Grau-S\'anchez et al. \cite{grau_adams}, any full LMM-based method with $s \geq 2$ has a higher convergence rate. As ODE solvers, full LMMs are typically not zero-stable and special choices of the coefficients have to be made. Employed as root-finders on the other hand, it turns out that LMMs are stable, due to the rapid decrease of the step size. This allows the usage of full LMMs that are otherwise not zero-stable.\par
Contrary to the Adams-type methods, the full LMM-based root-finders can achieve the convergence rate of $3$ in the limit that all history points are used. The $s=2$ and $s=3$ methods, $s$ being the number of history points, were explicitly constructed and provide a convergence rate of $2.73$ and $2.91$, respectively. Numerical experiments confirm these predicted convergence rates. Furthermore, application to pathological functions where Newton's method diverges show that the LMM-based methods also have enhanced stability properties.\par
Finally, we have implemented a robust LMM-based method that is guaranteed to converge when provided with an enclosing bracket of the root. The resulting robust LMM root-finder algorithm is a cascade of twelve root-finders increasing in convergence rate but decreasing in reliability. At the base sits bisection, so that the method is indeed guaranteed to converge to the root. At the top resides the $s=3$ LMM-based root-finder, providing a maximal convergence rate of $2.91$.\par
In terms of efficiency index, Brent's method is theoretically the preferred choice with $1.84$ compared to $1.71$ for the LMM-based method. However, numerical examples show that the increased convergence rate leads to a significant decrease in the total number of function evaluations over a range of test functions. Therefore, in practical situations, provided the derivative is available, the LMM-based method performs better.

\section*{Acknowledgements}
This work was generously supported by Philips Lighting and the Intelligent Lighting Institute.


\end{document}